\documentclass[11pt,reqno,twoside]{amsart}
\usepackage{amsmath}
\usepackage{amssymb}
\usepackage{amsthm}
\usepackage{latexsym}
\usepackage{amscd}
\usepackage{xypic}
\xyoption{curve}
\usepackage{ifthen} 
\usepackage{hyperref} 
\usepackage{graphicx}
\usepackage{enumerate}
\usepackage{enumitem}
\usepackage{rotating}

\numberwithin{equation}{section}

 \def\cocoa{{\hbox{\rm C\kern-.13em o\kern-.07em C\kern-.13em o\kern-.15em A}}}

\usepackage{xcolor}

\newtheorem{theorem}{Theorem}[section]

\newtheorem{proposition}[theorem]{Proposition}
\newtheorem{corollary}[theorem]{Corollary}

\theoremstyle{definition}
\newtheorem{remark}[theorem]{Remark}

\newtheorem{example}[theorem]{Example}

\newtheorem{question}[theorem]{Question}

\newcommand {\Hom}{\mathcal{H}\kern -0.25ex{\mathit om}}
\newcommand {\Ext}{\mathcal{E}\kern -0.25ex{\mathit xt}}
\newcommand {\im}{\mathrm{im}}
\newcommand {\rk}{\mathrm{rk}}

\newcommand {\Hilb}{\mathcal{H}\kern -0.25ex{\mathit ilb\/}}

\newcommand {\field}{k}

\newcommand {\cA}{\mathcal{A}}
\newcommand {\cB}{\mathcal{B}}

\newcommand {\bZ}{\mathbb{Z}}

\newcommand {\bC}{\mathbb{C}}
\newcommand {\bP}{\mathbb{P}}

\newcommand {\bF}{\mathbb{F}}
\newcommand{\cC}{{\mathcal C}}

\newcommand{\cS}{{\mathcal S}}
\newcommand{\cE}{{\mathcal E}}
\newcommand{\cF}{{\mathcal F}}
\newcommand{\cM}{{\mathcal M}}

\newcommand{\cO}{{\mathcal O}}
\newcommand{\cG}{{\mathcal G}}
\newcommand{\cI}{{\mathcal I}}

\newcommand{\Bl}{\operatorname{Bl}}
\newcommand{\Pic}{\operatorname{Pic}}
\newcommand{\Num}{\operatorname{Num}}

\def\p#1{{\bP^{#1}}}

\def\mapright#1{\mathbin{\smash{\mathop{\longrightarrow}
\limits^{#1}}}}

\title[Examples of surfaces which are Ulrich--wild]{Examples of surfaces which are Ulrich--wild}



\author[Gianfranco Casnati]{Gianfranco Casnati}
\thanks{The author is a member of GNSAGA group of INdAM and is supported by the framework of PRIN 2015 \lq Geometry of Algebraic Varieties\rq, cofinanced by MIUR}

\begin{document}

\begin{abstract}
We give examples of surfaces which are Ulrich--wild, i.e. that support families of dimension $p$ of pairwise non--isomorphic, indecomposable, Ulrich bundles for arbitrary large $p$.
\end{abstract}

\maketitle
\section{Introduction and Notation}
Throughout the whole paper $\field$ will denote an algebraically closed field and $\p N$ the projective space over $\field$ of dimension $N$. 

If $X\subseteq\p N$ is a {\sl variety}, i.e. an integral closed subscheme, then the study of coherent sheaves on $X$ is an important tool for understanding its geometric properties. From a cohomological viewpoint, the simplest sheaves on such an $X$ are the {\sl Ulrich} ones with respect to the very ample line bundle $\cO_X(h_X):=\cO_{\p N}(1)\otimes\cO_X$. There are several equivalent characterizations for such sheaves (e.g. see Proposition 2.1 of \cite{E--S--W}). In this paper we call the sheaf $\cF$  Ulrich if 
$$
h^i\big(X,\cF(-ih_X)\big)=h^j\big(X,\cF(-(j+1)h_X)\big)=0
$$
for each $i>0$ and $j<\dim(X)$. Ulrich sheaves are aCM, i.e. $h^i\big(X,\cF(th_X)\big)=0$ for $0< i<\dim(X)$ and each $t\in\bZ$: thus they are vector bundles when restricted to the smooth part of $X$.

In \cite{E--S--W}, the authors asked the following questions.

\begin{question} 
\label{ESW}
Is every variety (or even scheme) $X\subseteq\p N$ the support of an Ulrich sheaf? If so, what is the smallest possible rank for such a sheaf?
\end{question}

At present, answers to the questions above are known in a number of particular cases: e.g.,  see \cite{A--C--MR}, \cite{A--F--O}, \cite{Bea}, \cite{Bea1}, \cite{C--H2}, \cite{Cs4}, \cite{Cs5}, \cite{C--G}, \cite{C--K--M2}, \cite{C--MR}, \cite{MR--PL1}, \cite{MR--PL}, \cite{MR--PL2}, \cite{PL--T}.

The existence of many Ulrich sheaves on a variety $X$ can be viewed as a sign of the complexity of the variety itself. For example one could ask if $X\subseteq\p N$ is of {\sl Ulrich--wild representation type}, i.e. if it supports families of dimension $p$ of pairwise non--isomorphic, indecomposable, Ulrich sheaves for arbitrary large $p$. 

Ulrich--wildness is known for several classes of varieties. 
The case of {\sl surfaces}, i.e. smooth integral projective varieties of dimension $2$ is of particular interest. In this paper we prove the following result ($K_S$ denotes the canonical class of $S$).

\begin{theorem}
\label{tSurfaceWild}
Let $S$ be a surface endowed with a very ample line bundle  $\cO_S(h_S)$ and let $d:=h_S^2$.

If $S$ supports an Ulrich bundle with respect to $\cO_S(h_S)$ and 
\begin{equation}
\label{Bound1}
d^2+4(\chi(\cO_S)-2)d-{(h_SK_S)^2}>0,
\end{equation}
then $S$ is Ulrich--wild with respect to $\cO_S(h_S)$.
\end{theorem}

Unfortunately, the above result is not sharp. E.g. if $S$ is a del Pezzo surface, $h_SK_S=-d$ and $\chi(\cO_S)=1$, thus the first member of Inequality \eqref{Bound1} is strictly negative. Nevertheless del Pezzo surfaces are Ulrich--wild as shown in \cite{PL--T} and \cite{MR--PL}: see also Section 5.3 of \cite{F--PL} where the authors prove the Ulrich--wildness of each $2$--dimensional maximally del Pezzo variety (see \cite{B--S} and the references therein for details about such varieties). 

In particular, Theorem \ref{tSurfaceWild} yields some interesting examples. The first one are complete intersection surfaces of degree $4\le d\le 9$ in $\p N$ (see Example \ref{eCI}). As a consequence of results by A. Beauville and D. Faenzi, we show in Example \ref{eAbelian} that {\sl abelian surfaces} (i.e. surfaces $S$ with $K_S=0$ and $q(S)=2$) and {\sl $K3$ surfaces} (i.e., surfaces $S$ with $K_S = 0$ and $q(S) = 0$) are Ulrich--wild. A similar result has been recently proved in \cite{Fa} for {\sl $K3$ surfaces} (i.e. surfaces $S$ with $K_S=0$ and $q(S)=0$) extending the previous works \cite{C--K--M2} and \cite{A--F--O}. Thus each minimal surface $S$ with Kodaira dimension $\kappa(S)=0$ is Ulrich--wild (see Proposition 6 of \cite{Bea3}).

We also deal with {\sl surfaces of general type}  (i.e. minimal surfaces $S$ with $\kappa(S)=2$: see Corollary \ref{cGeneralType}) such that $\cO_S(h_S)\cong\cO_S(\lambda K_S)$ for some positive $\lambda\in\bZ$.

In order to find other examples, we prove the following helpful theorem: it makes slightly more precise the results from \cite{H--U--B}.

\begin{theorem}
\label{tUlrichTensor}
Let $X, Y\subseteq\p N$ be varieties endowed with Ulrich sheaves $\cA$ and $\cB$ respectively. 

If $X$ and $Y$ intersect properly,
$\cA$ and $\cB$ are locally free of respective ranks $a$ and $b$ along $X\cap Y$ and $Y$ is locally complete intersection at the points of  $X$, then  $\cA\otimes_{\p N}\cB$ is an Ulrich bundle of rank $ab$ on $X\cap Y$.
\end{theorem}

Since $\cO_{\p n}$ is the unique Ulrich bundle on $\p n$ with respect to $\cO_{\p n}(1)$, it follows that the above Theorem generalizes the well--known obvious fact that the restriction of an Ulrich sheaf to a general linear section of a variety is still Ulrich. 
One can construct some other applications besides hyperplane sections. 

E.g. in \cite{B--C--P} the authors considered the intersection of two varieties $G_1,G_2\subseteq\p9$ projectively isomorphic to the Pl\"ucker model of the grassmannian $G(2,5)$ of lines in $\p4$. If $G_1$ and $G_2$ are general, then  $F:=G_1\cap G_2$ is a smooth threefold which is {\sl Calabi--Yau}, i.e. $\cO_F(K_F)\cong\cO_F$ and $h^1\big(F,\cO_F\big)=0$: $F$ is called a  {\sl GPK$^{\mathit 3}$ threefold}. In Example \ref{eGPK3} we show that $F$  supports Ulrich bundles by applying Theorem \ref{tUlrichTensor} above, though $F$ is not a hypersurface section.

A second application is the construction of Ulrich bundles on hypersurface sections of del Pezzo threefolds (see Example \ref{eDelPezzo}).

We end the paper by collecting what we know about surfaces of degree up to $8$. 
We call  {\sl special} each Ulrich bundle $\cE$ on $S$ of even rank such that
$$
c_1(\cE)=\frac{\rk(\cE)}2(3h+K_S).
$$
If $\cE$ is any Ulrich bundle on $S$, then $\cE\oplus\cE^\vee(3h_S+K_S)$ is special: thus $S$ supports Ulrich bundles if and only if it supports special ones. The existence of a special Ulrich bundle $\cE$ on $S$ has some interesting consequences: e.g. if $\rk(\cE)=2$, then the Chow form of $S$ is pfaffian.

We then prove the following results over the complex field $\bC$ (if $S\subseteq\p N$, then $\pi(S)$ is the genus of a general hyperplane section of $S$).

\begin{theorem}
\label{tLowDegree}
Let $S\subseteq\p N$ be a surface of degree $d\le8$ and $\kappa(S)\ne1$.

Then $S$ supports special Ulrich bundles of rank $r_{Ulrich}^{sp}$ as indicated in the Table A.
\end{theorem}

\begin{theorem}
\label{tLowWild}
Let $S\subseteq\p N$ be a surface of degree $d\le 8$ and $\kappa(S)\ne1$.

Then $S$ is Ulrich--wild if and only if either $d\ge5$, or $d\le4$ and  $\pi(S)\ge1$.
\end{theorem}

When $\kappa(S)=1$  we are not able to prove or disprove the existence of Ulrich bundles, though \cite{MR--PL2} provides some evidences in this direction.

In Section \ref{sGeneral} we list some results about Ulrich bundles. In Section \ref{sSurfaceWild} we prove Theorem \ref{tSurfaceWild}. In Section \ref{sUlrichTensor} we prove Theorem \ref{tUlrichTensor}.
In Section \ref{sLowDegree} we prove Theorems \ref{tLowDegree} and \ref{tLowWild}. 

For reader's benefit we incorporate in Table A below all the important informations about surfaces in $\p N$ of degree up to $8$. Their classification can be essentially found in \cite{Io2} and \cite{Io3}: see also \cite{Ok3} and the references therein. In \cite{Io3} and \cite{Ok3} the authors ruled out the case of irregular surfaces of degree $8$ in $\p4$ with different incomplete arguments. Such a case is completely described in the proof of Proposition 2.1 of \cite{A--D--S}: see also \cite{Ra}, Section 1.3 and Lemma 1.4 for some further details.

\begin{remark}
\label{rTable}
In Table A we use the following notation.

\begin{itemize}
\item $X_{d_1,\dots, d_{N-2}}$ denotes any complete intersection of hypersurfaces of degrees $d_1,\dots, d_{N-2}$ in $\p N$. In particular there is a natural rational map from a non--empty open subset of a Segre product $\prod_{i=1}^{N-2}\bP\left(H^0\big(\p N,\cO_{\p N}(d_i)\big)\right)$ to the Hilbert scheme whose image corresponds to the locus of points representing such surfaces.
\item Consider a ruled surface $X\cong\bP(\mathcal H)$, where $\mathcal H$ is a rank $2$ vector bundle on a curve $C$ which is normalized in the sense of Section V.2 of \cite{Ha2}. Then $p\colon X\to C$ denotes the canonical map, $e:=-\deg(\det(\mathcal H))$ the invariant of $X$, $\xi$ the class of the divisor $\cO_X(1)$ and $\mathfrak a f$ the class of the pull back of the divisor $\mathfrak a$ on the curve $C$ via $p$: if $a:=\deg(\mathfrak a)$ we will also write $a f$ for its numerical class in $\Num(X)$. If $C\cong\p1$ we set $\bF_e:=\bP(\cO_{\p1}\oplus\cO_{\p1}(-e))$, for each $e\ge0$. 
\item $\Bl_{P_1,\dots, P_t}X$ denotes the blow up of $X$ at the points $P_1,\dots, P_t$: in this case $\sigma$ denotes the blow up map, $e_i:=\sigma^{-1}(P_i)$.
\item The number  $r_{Ulrich}^{sp}$, if any, denotes the minimal rank of a special (not necessarily indecomposable) Ulrich bundle on $S$: in the case $S$ is also known to support or not Ulrich line bundles \lq$\exists$ line bundles\rq\ and \lq no line bundles\rq\ are added here. 
For  the details see the proof of Theorem \ref{tLowDegree}.

In two cases, denoted by the sentence \lq no results\rq\  in the last column in the table, we are unable to prove or disprove the existence of Ulrich bundles on $S$: these are exactly the cases when $\kappa(S)=1$ (see Remark \ref{kappa=1}).

\end{itemize}

\vfill\eject

\begin{small}
\centerline{\text{\it{Table A: Surfaces of degree $d\le8$ in $\p N$}}}
\vglue3truemm
\centerline{
\begin{sideways}
$\begin{array}{cccccccccc}        
\text{Class} &\text{Abstract model of $S$, $\cO_S(h_S)$} & d & \kappa &  p_g & q & K_S^2 &  h_SK_S & N & r_{Ulrich}^{sp}  \\ \hline \hline
\text{(I)} &X_{1}\cong\p2,\ \cO_{\p2}(1) & 1 & -\infty  & 0 & 0 & 9 & -3 & 2 & 2,\ \text{$\exists$ line bundles}\\ \hline
\text{(II)} &X_{2}\cong\p1\times\p1,\ \cO_{\p1}(1)\boxtimes\cO_{\p1}(1)& 2 & -\infty & 0 & 0 & 8 & -4 & 3 & 2,\ \text{$\exists$ line bundles} \\ \hline
\text{(III)} &X_{3}\cong\Bl_{P_1,\dots, P_6}\p2,\ \sigma^*\cO_{\p2}(3)\otimes\cO_S\left(-\sum_{i=1}^6e_i\right) & 3 & -\infty & 0 & 0 & 3 & -3 & 3 & 2,\ \text{$\exists$ line bundles}  \\ \hline
\text{(IV)} &\bF_1,\ \cO_S\left(\xi+2f\right)  & 3 & -\infty & 0 & 0 & 8 & -5 & 4 & 2,\ \text{$\exists$ line bundles}   \\ \hline
\text{(V)} &X_{2,2}\cong \Bl_{P_1,\dots,P_5}\p2,\ \sigma^*\cO_{\p2}(3)\otimes\cO_S\left(-\sum_{i=1}^5e_i\right) & 4 & -\infty & 0 & 0 & 4 & -4 & 4 & 2,\ \text{$\exists$ line bundles}\\ \hline
\text{(VI)} &\bF_e, \ e=0,2,\ \cO_S\left(\xi+\frac{e+4}2f\right) & 4 & -\infty & 0 & 0 & 8 & -6 & 5 & 2,\ \text{$\exists$ line bundles} \\ \hline
\text{(VII)} &\p2,\ \cO_{\p2}(2) & 4 & -\infty & 0 & 0 & 9 & -6 & 5 & 2,\ \text{no line bundles}    \\ \hline
\text{(VIII)} &X_{4},\ \cO_{\p3}(1)\otimes\cO_S & 4 & 0 &  1 & 0 & 0 & 0 & 3  & 2   \\ \hline
\text{(IX)} &\Bl_{P_1,\dots,P_{8}}\p2,\ \sigma^*\cO_{\p2}(4)\otimes\cO_S\left(-2e_1-\sum_{i=2}^{8}e_i\right)& 5 & -\infty & 0 & 0 & 2 & -3 & 4  & 2,\ \text{$\exists$ line bundles}  \\ \hline
\text{(X)} &\Bl_{P_1,\dots,P_4}\p2,\ \sigma^*\cO_{\p2}(3)\otimes\cO_S\left(-\sum_{i=1}^4e_i\right)& 5 & -\infty & 0 & 0 & 5 & -5 & 5 & 2,\ \text{$\exists$ line bundles}  \\ \hline
\text{(XI)} &\bF_e, \ e=1,3,\ \cO_S\left(\xi+\frac{e+5}2f\right) & 5 & -\infty & 0 & 0 & 8 & -7 & 6 & 2,\ \text{$\exists$ line bundles}  \\ \hline
\text{(XII)} &\text{elliptic ruled surface with $e=-1$},\ \cO_S\left(\xi+2f\right) & 5 & -\infty & 0 & 1 & 0 & -5 & 6 & 2,\ \text{$\exists$ line bundles}  \\ \hline
\text{(XIII)} &X_{5},\ \cO_{\p3}(1)\otimes\cO_S & 5 & 2 & 4 & 0 & 5 & 5 & 3 & \text{$\gg0$,  generically $\le2$} \\ \hline
\text{(XIV)} &\Bl_{P_1,\dots,P_{10}}\p2,\ \sigma^*\cO_{\p2}(4)\otimes\cO_S\left(-\sum_{i=1}^{10}e_i\right) & 6 & -\infty & 0 & 0 & -1 & -2 & 4 & 2,\ \text{$\exists$ line bundles}  \\ \hline
\text{(XV)} &\Bl_{P_1,\dots,P_3}\p2,\ \sigma^*\cO_{\p2}(3)\otimes\cO_S\left(-\sum_{i=1}^3e_i\right)& 6 & -\infty & 0 & 0 & 6 & -6 & 6 & 2,\ \text{$\exists$ line bundles}  \\ \hline
\text{(XVI)} &\bF_e, \ e=0,2,4,\ \cO_S\left(\xi+\frac{e+6}2f\right) & 6 & -\infty & 0 & 0 & 8 & -8 & 7 & 2,\ \text{$\exists$ line bundles}   \\ \hline
\text{(XVII)} &\text{elliptic ruled surface with $e=0$},\ \cO_S\left(\xi+3f\right) & 6 & -\infty & 0 & 1 & 0 & -6 & 5 & 2,\ \text{$\exists$ line bundles}  \\ \hline
\text{(XVIII)} &X_{2,3},\ \cO_{\p4}(1)\otimes\cO_S & 6 & 0 & 1 & 0 & 0 & 0 & 4 & \text{$2$}    \\ \hline
\text{(XIX)} &X_{6},\ \cO_{\p3}(1)\otimes\cO_S & 6 & 2 &  10 & 0 & 24 &12 & 3 & \text{$\gg0$,  generically $\le2$}  \\ \hline
\text{(XX)} &\Bl_{P_1,\dots,P_{11}}\p2,\ \sigma^*\cO_{\p2}(6)\otimes\cO_S\left(-2\sum_{i=1}^{6}e_i-\sum_{j=7}^{11}e_j\right)& 7 & -\infty & 0 & 0 & -2 & -1 & 4 &2   \\ \hline
\text{(XXI)} &\Bl_{P_1,\dots,P_{8}}\p2,\ \sigma^*\cO_{\p2}(6)\otimes\cO_S\left(-2\sum_{i=1}^{7}e_i-e_8\right)& 7 & -\infty & 0 & 0 & 1 & -3 & 5 &2  \\ \hline
\text{(XXII)} &\Bl_{P_1,\dots,P_{9}}\p2,\ \sigma^*\cO_{\p2}(4)\otimes\cO_S\left(-\sum_{i=1}^{9}e_i\right)& 7 & -\infty & 0 & 0 & 0 & -3 & 5 &2 \\ \hline
\text{(XXIII)} &{\Bl_{P_1,\dots,P_{9}}\bF_e,\ e=0,\dots,3,}  & & & & & & & & \\ 
&{\sigma^*\cO_{\bF_e}(2\xi+(4+e)f)\otimes\cO_S\left(-\sum_{i=1}^{9}e_i\right)}& 7 & -\infty & 0 & 0 & -1 & -3 & 5 & 2  \\ \hline
\text{(XXIV)} &\Bl_{P_1,\dots,P_{6}}\p2,\ \sigma^*\cO_{\p2}(4)\otimes\cO_S\left(-2e_1-\sum_{i=2}^{6}e_i\right)& 7 & -\infty & 0 & 0 & 3 & -5 & 6 &2 \\ \hline
\text{(XXV)} &\Bl_{P_1,P_2}\p2,\ \sigma^*\cO_{\p2}(3)\otimes\cO_S\left(-\sum_{i=1}^2e_i\right)& 7 & -\infty & 0 & 0 & 7 & -7 & 7 & 2,\ \text{$\exists$ line bundles} \\ \hline
\text{(XXVI)} &\bF_e, \ e=1,3,5,\ \cO_S\left(\xi+\frac{e+7}2f\right) & 7 & -\infty & 0 & 0 & 8 & -9 & 8 & 2,\ \text{$\exists$ line bundles} \\ \hline
\end{array}$
\end{sideways}
}

\centerline{\text{\it{Table A: Surfaces of degree $d\le8$ in $\p N$}}}
\vglue3truemm
\centerline{
\begin{sideways}
$\begin{array}{cccccccccc}        
\text{Class} &\text{Abstract model of $S$, $\cO_S(h_S)$} & d & \kappa &  p_g & q & K_S^2 &  h_SK_S & N & r_{Ulrich} \\ \hline\hline
\text{(XXVII)} &\text{elliptic ruled surface with $e=-1, 1$},\ \cO_S\left(\xi+(4+\lceil \frac e2\rceil)f\right) & 7 & -\infty & 0 & 1 & 0 & -7 & 8 & 2,\ \text{$\exists$ line bundles} \\ \hline
\text{(XXVIII)} & \Bl_{P_1}X_{2,2,2},\ \sigma^*\cO_X(h_X)\otimes\cO_S\left(-e_1\right)& 7 & 0 & 1 & 0 & -1 & 1 & 4 &  2  \\ \hline
\text{(XXIX)} &\text{Proper elliptic},\ \cO_{\p4}(1)\otimes\cO_S & 7 & 1 & 2 & 0 & 0 & 3 & 4  & \text{no results} \\ \hline
\text{(XXX)} &X_{7} & 7 & 2 & 20 & 0 & 63 & 21 &3  & \text{$\gg0$,  generically $2$}  \\ \hline
\text{(XXXI)} &\Bl_{P_1,\dots,P_{16}}\p2,\ \sigma^*\cO_{\p2}(6)\otimes\cO_S\left(-2\sum_{i=1}^4e_i-\sum_{j=5}^{16}e_j\right)& 8 & -\infty & 0 & 0 & -7 & 2 & 4 & 2   \\ \hline
\text{(XXXII)} &\Bl_{P_1,\dots,P_{12}}\bF_e,\ e=0,\dots,4,  & & & & & & & & \\  
&\sigma^*\cO_{\bF_e}(2\xi+(5+e)f)\otimes\cO_S\left(-\sum_{i=1}^{12}e_i\right)& 8 & -\infty & 0 & 0 & -4 & -2 & 5 & 2  \\ \hline
\text{(XXXIII)} &\Bl_{P_1,\dots,P_{10}}\p1\times\p1,\ \sigma^*(\cO_{\p1}(3)\boxtimes\cO_{\p1}(3))\otimes\cO_S\left(-\sum_{i=1}^{10}e_i\right)& 8 & -\infty & 0 & 0 & -2 & -2 & 5 & 2 \\ \hline
\text{(XXXIV)} &\Bl_{P_1,\dots,P_{11}}\p2,\  \sigma^*\cO_{\p2}(7)\otimes\cO_S\left(-2\sum_{i=1}^{10}e_i-e_{11}\right)& 8 & -\infty & 0 & 0 & -2 & 0 & 4 & 2  \\ \hline
\text{(XXXV)} &\Bl_{P_1,\dots,P_{10}}\p2,\  \sigma^*\cO_{\p2}(6)\otimes\cO_S\left(-2\sum_{i=1}^6e_i-\sum_{j=7}^{10}e_j\right)& 8 & -\infty & 0 & 0 & -1 & -2 & 5 & 2 \\ \hline
\text{(XXXVI)} &\Bl_{P_1,\dots,P_{8}}\bF_e,\ e=0,\dots,3,  & & & & & & & &  \\   
&\sigma^*\cO_{\bF_e}(2\xi+(4+e)f)\otimes\cO_S\left(-\sum_{i=1}^{8}e_i\right)& 8 & -\infty & 0 & 0 & 0 & -4 &  6 & 2  \\ \hline
\text{(XXXVII)} &\Bl_{P_1,\dots,P_{8}}\p2,\ \sigma^*\cO_{\p2}(4)\otimes\cO_S\left(-\sum_{i=1}^{8}e_i\right)& 8 & -\infty & 0 & 0 & 1 & -4 & 6 & 2  \\ \hline
\text{(XXXVIII)} &\Bl_{P_1,\dots,P_{7}}\p2,\ \sigma^*\cO_{\p2}(6)\otimes\cO_S\left(-2\sum_{i=1}^{7}e_i\right)& 8 & -\infty & 0 & 0 & 2 & -4 & 6 & 2  \\ \hline
\text{(XXXIX)} &\Bl_{P_1,\dots,P_{5}}\p2,\ \sigma^*\cO_{\p2}(4)\otimes\cO_S\left(-2e_1-\sum_{i=2}^{5}e_i\right)& 8 & -\infty & 0 & 0 & 4 & -6 & 7 & 2  \\ \hline
\text{(XL)} &\bF_e, \ e=0,2,4,6,\ \cO_S\left(\xi+\frac{e+8}2f\right) & 8 & -\infty & 0 & 0 & 8 & -10 & 9 & 2,\ \text{$\exists$ line bundles}   \\ \hline
\text{(XLI)} &\Bl_{P_1}\p2,\ \sigma^*\cO_{\p2}(3)\otimes\cO_S\left(-e_1\right)& 8 & -\infty & 0 & 0 & 8 & -8 & 8 & 2,\ \text{no line bundles} \\ \hline
\text{(XLII)} &\p1\times\p1,\ \cO_{\p1}(2)\boxtimes\cO_{\p1}(2)& 8 & -\infty & 0 & 0 & 8 & -8 & 8 & 2,\ \text{$\exists$ line bundles}    \\ \hline
\text{(XLIII)} &\text{elliptic ruled surface with $e=0, 2$, $\cO_S\left(\xi+(4+\frac e2)f\right)$} & 8 & -\infty & 0 & 1& 0 & -8 & 7 & 2,\ \text{$\exists$ line bundles} \\ \hline
\text{(XLIV)} &\text{$\Bl_{P_1,\dots,P_8}X$, $X$ elliptic ruled surface with $e=-1,1$,} & & & & & & & & \\ 
&\text{$\sigma^*\cO_X(2\xi+(4+e)f)\otimes\cO_S\left(-\sum_{i=1}^8e_i\right)$} & 8 & -\infty & 0 & 1& -8 & 0 & 4 & 2  \\ \hline
\text{(XLV)} &\text{elliptic ruled surface with $e=-1$, $\cO_S\left(2\xi+f\right)$}& 8 & -\infty & 0 & 1 & 0 & 4 & 5 & 2,\ \text{no line bundles}  \\ \hline
\text{(XLVI)} &\text{scroll with $e=-2$ over a curve $C$ of genus 2, $\cO_S\left(\xi+3f\right)$} & 8 & -\infty & 0 & 2 & -8 & -6 & 5 & 2,\ \text{$\exists$ line bundles}  \\ \hline
\text{(XLVII)} &\text{$C\times\p1$, $C\subseteq\p2$ smooth with $\deg(C)=4$, $\cO_C(1)\boxtimes\cO_{\p1}(1)$}& 8 & -\infty & 0 & 3 & -16 & -4 & 5 & 2,\ \text{$\exists$ line bundles}  \\ \hline
\text{(XLVIII)} & \text{$\Bl_{P_1}X$, $X\subseteq\p7$ $K3$ surface, $\sigma^*\cO_X(h_X)\otimes\cO_S\left(-2e_1\right)$}& 8 & 0 & 1& 0 & -1 &2 & 4 & 2  \\ \hline
\text{(IL)} &K3,\ \cO_{\p5}(1)\otimes\cO_S& 8 & 0 & 1 & 0 & 0 & 0 & 5 & \text{$2$}  \\ \hline
\text{(L)} &\text{Proper elliptic},\ \cO_{\p4}(1)\otimes\cO_S & 8 & 1 & 2 & 0 & 0 & 4 & 4 &  \text{no results}     \\ \hline
\text{(LI)} &X_{2,4},\ \cO_{\p4}(1)\otimes\cO_S & 8 & 2 & 5 & 0 & 8 & 8 & 4 & \text{$\gg0$,  generically $\le4$}  \\ \hline
\text{(LII)} &X_{8},\ \cO_{\p3}(1)\otimes\cO_S & 8 & 2 & 35 & 0 & 128 & 32 & 3 & \text{$\gg0$,  generically $2$}   \\ \hline
\end{array}$
\end{sideways}
}

\end{small}

\medbreak
\end{remark}

\section{General results}
\label{sGeneral}
Let $\cE$ be an Ulrich sheaf on a variety $X\subseteq\p N$ with $\dim(X)=n$. As pointed out in Section 2 of \cite{E--S--W} (see in particular Proposition 2.1 therein) we know that $\cE$ is aCM and $E:=\bigoplus_{t\in\bZ}H^0\big(X,\cE(th_X)\big)$ has a minimal free resolution over $P:=\field[x_0,\dots,x_N]$ of the form 
\begin{equation}
\label{seqLinear}
\begin{aligned}
0\longrightarrow P(n-N)^{\oplus\alpha_{N-n}}&\longrightarrow P(n-N+1)^{\oplus\alpha_{N-n-1}}\longrightarrow\dots\\
&\longrightarrow P(-1)^{\oplus\alpha_1}\longrightarrow P^{\oplus\alpha_0}\longrightarrow E\longrightarrow0.
\end{aligned}
\end{equation}

Ulrich bundles also behave well with respect to the notions of (semi)stability and $\mu$--(semi)stability. For each coherent sheaf $\cF$ on $X$, the slope $\mu(\cF)$ and the reduced Hilbert polynomial $p_{\cF}(t)$ (with respect to $\cO_X(h_X)$) are defined as follows:
$$
\mu(\cF)= c_1(\cF)h_X^{\dim (X)-1}/\rk(\cF), \qquad p_{\cF}(t)=\chi(\cF(th_X))/\rk(\cF).
$$
The coherent sheaf $\cF$ is {\sl $\mu$--semistable} (resp. {\sl $\mu$--stable}) if for all subsheaves
$\mathcal G$ with $0<\rk(\mathcal G)<\rk(\cF)$ we have $\mu(\mathcal G) \le \mu(\cF)$ (resp. $\mu(\mathcal G)< \mu(\cF)$).

The coherent sheaf $\cF$ is called {\sl semistable} (resp. {\sl stable}) if for all proper non--zero subsheaves $\mathcal G$, $p_{\mathcal G}(t) \le  p_{\cF}(t)$ (resp. $p_{\mathcal G}(t) <  p_{\cF}(t)$) for $t\gg0$. We recall that in order to check the semistability and stability of $\cF$ one can restrict the attention to subsheaves with torsion--free quotient. The following chain of implications holds for $\cF$:
$$
\text{$\cF$ is $\mu$--stable}\ \Rightarrow\ \text{$\cF$ is stable}\ \Rightarrow\ \text{$\cF$ is semistable}\ \Rightarrow\ \text{$\cF$ is $\mu$--semistable.}
$$
We revert below some of the above implications.

\begin{theorem}
\label{tUnstable}
Let $X\subseteq\p N$ be a smooth variety.
If $\cE$ is a Ulrich bundle on $X$ the following assertions hold.
\begin{enumerate}
\item $\cE$ is semistable and $\mu$--semistable.
\item $\cE$ is stable if and only if it is $\mu$--stable.
\item If
\begin{equation*}
\label{seqUnstable}
0\longrightarrow\mathcal \cG\longrightarrow\cE\longrightarrow\mathcal H\longrightarrow0
\end{equation*}
is an exact sequence of coherent sheaves with $\mathcal H$ torsion free and $\mu(\mathcal G)=\mu(\cE)$, then both $\mathcal G$ and $\mathcal H$ are Ulrich bundles.
\end{enumerate}
\end{theorem}
\begin{proof}
See Theorem 2.9 of \cite{C--H2}.
\end{proof}

Thus, Ulrich bundles on $X$ of minimal rank (if any) are both stable and $\mu$--stable. It is interesting to estimate
the size of the families of Ulrich bundles.

\begin{theorem}
\label{tFPL}
Let $X$ be a smooth variety endowed with a very ample line bundle $\cO_X(h_X)$.
If $\cA$ and $\cB$ are simple Ulrich bundles on $X$ such that $h^1\big(X,\cA\otimes\cB^\vee\big)\ge3$ and every non--zero morphism $\cA\to\cB$ is an isomorphism, then $X$ is Ulrich--wild.
\end{theorem}
\begin{proof}
See \cite[Theorem A, Corollary 2.1 and Remark 1.6 iii)]{F--PL}. Indeed $\mathcal A$ and $\mathcal B$, being Ulrich, satisfy $p_{\mathcal A}(t)=p_{\mathcal B}(t)$ by \cite[Lemma 2.6]{C--H2} and are semistable by \cite[Theorem 2.9]{C--H2}.
\end{proof}

From now on we restrict our attention to Ulrich bundles $\cE$ on a surface $S$. The Serre duality for $\cE$ is
$h^i\big(S,\cE\big)=h^{2-i}\big(S,\cE^\vee(K_S)\big)$, for $i=0,1,2$.
The Riemann--Roch theorem  for $\cE$ is
\begin{equation}
\label{RRGeneral}
h^0\big(S,\cE\big)+h^{2}\big(S,\cE\big)=\rk(\cE)\chi(\cO_S)+\frac{c_1(\cE)(c_1(\cE)-K_S)}2-c_2(\cE)+h^{1}\big(S,\cE\big),
\end{equation}
where $\chi(\cO_S):=1-q(S)+p_g(S)$.

\begin{proposition}
\label{pUlrich}
Let $S$ be a surface endowed with a very ample line bundle  $\cO_S(h_S)$ and let $d:=h_S^2$.
If $\cE$ is a vector bundle on $S$, then the following assertions are equivalent:
\begin{enumerate}
\item $\cE$ is an Ulrich bundle;
\item $\cE^\vee(3h_S+K_S)$ is an Ulrich bundle;
\item $\cE$ is an aCM bundle and 
\begin{equation}
\label{eqUlrich}
\begin{gathered}
c_1(\cE)h_S=\frac{\rk(\cE)}2(3d+h_SK_S),\\ 
c_2(\cE)=\frac{1}2(c_1(\cE)^2-c_1(\cE)K_S)-\rk(\cE)(d-\chi(\cO_S));
\end{gathered}
\end{equation}
\item $h^0\big(S,\cE(-h_S)\big)=h^0\big(S,\cE^\vee(2h_S+K_S)\big)=0$ and Equalities \eqref{eqUlrich} hold.
\end{enumerate}
\end{proposition}
\begin{proof}
See \cite[Proposition 2.1]{Cs4}.
\end{proof}

If $\cE$ is an Ulrich bundle on $S$, then the first Equality \eqref{eqUlrich} and the Hodge theorem applied to $h_S$ and $c_1(\cE)$ yields \begin{equation}
\label{HodgeGeneral}
c_1(\cE)^2\le\frac{\rk(\cE)^2}{4d}(3d+h_SK_S)^2.
\end{equation}

If $\cF$ is another Ulrich bundle on $S$ 
\begin{equation}
\label{CH2}
\chi(\cE^\vee\otimes\cF)=\rk(\cF)c_1(\cE)K_S-c_1(\cE)c_1(\cF)+\rk(\cE)\rk(\cF)(2d-\chi(\cO_S))
\end{equation}
(see Proposition 2.12 of \cite{C--H2} for the details).

Finally, the hypothesis $\rk(\cE)=1$ yields $c_2(\cE)=0$. Thus, if $\cE\cong\cO_S(D)$, then Equalities \eqref{eqUlrich} become
\begin{equation}
\label{eqLineBundle}
Dh_S=\frac12(3d+h_SK_S),\qquad D^2=2(d-\chi(\cO_S))+DK_S.
\end{equation}

\section{Ulrich--wildness of surfaces}
\label{sSurfaceWild}
In this section we first prove Theorem \ref{tSurfaceWild} stated in the introduction.

\medbreak
\noindent{\it Proof of Theorem \ref{tSurfaceWild}.}
Let $\cE$ be an Ulrich bundle of minimal rank $r$ on $S$: Theorem \ref{tUnstable} implies that $\cE$ is stable, hence simple (see \cite{H--L}, Corollary 1.2.8). 

The bundle  $\cF:=\cE^\vee(3h_S+K_S)$ is stable and simple too. In particular, every non--zero morphism $\cE\to\cF$ is an isomorphism, thanks to  \cite[Proposition 1.2.7]{H--L}. Moreover, $\cF$ is Ulrich by Proposition \ref{pUlrich}.

Since $c_1(\cF)=r(3h_S+K_S)-c_1(\cE)$, it follows from Equality \eqref{CH2} that
\begin{align*}
h^1\big(S,\cE^\vee\otimes\cF)&=-\chi(\cE^\vee\otimes\cF)+h^2\big(S,\cE^\vee\otimes\cF)\ge-\chi(\cE^\vee\otimes\cF)=\\
&=-rc_1(\cE)K_S+rc_1(\cE)(3h_S+K_S)-c_1(\cE)^2-r^2(2d-\chi(\cO_S)).
\end{align*}
The first Equality \eqref{eqUlrich} combined with Inequalities \eqref{HodgeGeneral} and \eqref{Bound1} yield
$$
h^1\big(S,\cE^\vee\otimes\cF)\ge\frac{r^2}{4d}(d^2+4\chi(\cO_S) d-(h_SK_S)^2)>2.
$$
Thus the statement follows from Theorem \ref{tFPL}.
\qed
\medbreak

\begin{example}
\label{eCI}
Thanks to  \cite{B--H} and \cite{H--U--B} each complete intersection surface is the support of an Ulrich bundle as well.

Let $S\subseteq\p N$ be a non--degenerate complete intersection surface of degree $d$. Then Theorem \ref{tSurfaceWild} implies that $S$ is certainly Ulrich--wild if either $5\le d\le 9$, or $N=3$ and $d=4$.  When $d=N=3$ (resp. $d=N=4$) the Ulrich--wildness of $S$ was proved in \cite{C--H2} (resp \cite{PL--T}). If $d\le 2$, then $S$ supports a finite number of indecomposable Ulrich (and aCM) bundles. If $d\ge10$ and $N=3$, $S$ does not satisfy Inequality \eqref{Bound1}, thus we cannot say anything about its Ulrich--wildness.
\end{example}

\begin{example}
\label{eAbelian}
Let $S$ be a minimal surface with $\kappa(S)=0$ over $\mathbb C$. The Enriques--Kodaira classification implies that $S$ is either abelian, or a $K3$ surface, or an Enriques surface, or a bielliptic surface.
The Riemann--Roch and the Kodaira vanishing theorems on $S$ imply that
$$
d=2h^0\big(S,\cO_S(h_S)\big)-2\chi(\cO_S)>4(2-\chi(\cO_S)). 
$$
Since trivially $h_SK_S=0$ in all these cases, we immediately deduce from Theorem \ref{tSurfaceWild}, \cite[Proposition 6]{Bea3} and \cite[Theorem 1]{Fa} that every such $S$ is Ulrich--wild.

The Ulrich--wildness of Enriques, $K3$ and bielliptic surfaces were already shown in \cite{Cs4}, \cite{Fa} and \cite{Cs5}.
\end{example}

We close this section with the following corollary of Theorem \ref{tSurfaceWild}, where $k=\bC$.

\begin{corollary}
\label{cGeneralType}
Let $S$ be a surface endowed with a very ample line bundle $\cO_S(h_S)$ such that $\cO_S(h_S)\cong\cO_S(\lambda K_S)$ for some positive integer $\lambda$.
If $S$ supports a bundle which is Ulrich with respect to $\cO_S(h_S)$, then $S$ is Ulrich--wild.
\end{corollary}
\begin{proof}
If $\cO_S(h_S):=\cO_S(\lambda K_S)$ is very ample, then $\cO_S(K_S)$ is ample. 
The ampleness of $\cO_S(K_S)$ and the Nakai criterion yield that $S$ is minimal.
Thus both $K_S^2$ and $\chi(\cO_S)$ are positive (see Theorem VII.1.1 of \cite{B--H--P--VV}).

Since the degree of $S$ is $d=\lambda^2K_S^2$ and $h_SK_S=\lambda K_S^2$, it follows that Inequality \eqref{Bound1} is equivalent to
\begin{equation}
\label{Bound2}
(\lambda^2-1)K_S^2+4\chi(\cO_S)>8,
\end{equation}
which is immediately satisfied if $\lambda\ge3$.

If $\lambda=2$, then $\cO_S(2K_S)$ is very ample, hence $h^0\big(S,\cO_S(2K_S)\big)\ge4$. Proposition VII.5.3 of \cite{B--H--P--VV} implies $K^2_S+\chi(\cO_S)\ge4$, whence we deduce Inequality \eqref{Bound2}.

If $\lambda=1$, then $\cO_S(K_S)$ is very ample, hence $p_g(S)\ge4$. In this case  Inequality \eqref{Bound2} is equivalent to $\chi(\cO_S)>2$. If $\chi(\cO_S)\le2$, then either $q(S)=p_g(S)$, or $q(S)+1=p_g(S)$. 

The Bogomolov--Miyaoka--Yau inequality (see Theorem VII.4.1 of \cite{B--H--P--VV}) and the equality $12\chi(\cO_S)=K_S^2+c_2(\Omega_S^1)$ imply $K_S^2\le 9\chi(\cO_S)$. By combining the latter inequality with Theorem 3.2 of \cite{De--Bea} we obtain $q(S)=p_g(S)=4$ if $\chi(\cO_S)=1$, and $6\ge q(S)+1=p_g(S)\ge4$ if $\chi(\cO_S)=2$. Since every surface in $\p3$ is regular, it follows that we can restrict to the case $6\ge q(S)+1=p_g(S)\ge5$. 

If $p_g(S)=5$, then the double point formula for $S$ (see \cite[Example A.4.1.3]{Ha2}) implies that $d=h_S^2=h_SK_S=K_S^2$ satisfies $d^2-17d+24=0$ which has no integral solutions. We deduce that $q(S)+1=p_g(S)=6$, hence $S$ should be the product of a curve of genus $2$ and a curve of genus $3$ (see the Theorem in the Appendix of \cite{De--Bea}): since the canonical map for such a surface has degree at least $2$, it follows a contradiction.

Thus, if $S$ supports an Ulrich bundle, then it is Ulrich--wild by Theorem \ref{tSurfaceWild}.
\end{proof}

\section{Ulrich bundles on intersections}
\label{sUlrichTensor}
In this section we prove  the following generalization of the main result of \cite{H--U--B}, which is usually stated only for complete intersections.

\medbreak
\noindent{\it Proof of Theorem \ref{tUlrichTensor}.}
Let $n:=\dim(X)$, $m:=\dim(Y)$: thanks to the hypothesis we have that $\dim(X\cap Y)=m+n-N$. If $m+n\le N$, then the statement is trivial, thus we will assume $m+n>N$ from now on.

There is an exact sequence of the form 
\begin{align*}
0\longrightarrow\cO_{\p N}(n-N)^{\oplus\alpha_{N-n}}&\longrightarrow\cO_{\p N}(n-N+1)^{\oplus\alpha_{N-n-1}}\longrightarrow\dots\\
&\longrightarrow\cO_{\p N}(-1)^{\oplus\alpha_1}\longrightarrow\cO_{\p N}^{\oplus\alpha_0}\longrightarrow\cA\longrightarrow0
\end{align*}
obtained by sheafifying Sequence \eqref{seqLinear} with $E:=\bigoplus_{t\in\bZ} H^0\big(X,\cA(th_X)\big)$.

The variety $Y$ intersects properly $X$ and it is locally complete intersection at each point $x\in X\cap Y$. Thus its local equations in $\cO_{\p N,x}$ are regular elements for $\cO_{X,x}$, hence also in $\cA_x\cong\cO_{X,x}^{\oplus a}$. When $x\in Y\setminus X$ then  $\cA_x$ is zero. Thus the above sequence tensored by $\cO_Y$ is  everywhere exact along $Y$, thanks to Theorem 7 of \cite{No}.

Taking into account that $\cA\otimes_{\p N}\cO_Y\otimes_Y\cB\cong\cA\otimes_{\p N}\cB$, by tensoring such a restricted sequence with $\cB$, we obtain the complex on $Y$
\begin{equation*}
\label{seqRestriction}
\begin{aligned}
0\longrightarrow\cB((n-N)h_Y)^{\oplus\alpha_{N-n}}&\mapright{\varphi_{N-n}}\cB((n-N+1)h_Y)^{\oplus\alpha_{N-n-1}}\mapright{\varphi_{N-n-1}}\dots\\
&\mapright{\varphi_2}\cB(-h_Y)^{\oplus\alpha_1}\mapright{\varphi_1}\cB^{\oplus\alpha_0}\longrightarrow\cA\otimes_{\p N}\cB\longrightarrow0.
\end{aligned}
\end{equation*}
Such a complex is actually exact on $X\cap Y$, because $\cB_x\cong\cO_{Y,x}^{\oplus b}$. When $x\in Y\setminus X$ it is still exact because $\cA_x$ is zero.

Since $\cB$ is Ulrich on $Y$, it follows that 
\begin{equation}
\label{BUlrich}
h^i\big(Y,\cB(-ih_{Y})\big)=h^j\big(Y,\cB(-(j+1)h_{Y})\big)=0,
\end{equation}
for $i>0$ and $j<m$ by definition (see the introduction). 

Let $\mathcal C_0:=\cA\otimes_{\p N}\cB$ and $\mathcal C_\lambda:=\im(\varphi_\lambda)$ for $1\le \lambda\le N-n$: notice that $\mathcal C_{N-n}=\cB((n-N)h_Y)^{\oplus\alpha_{N-n}}$. We have the exact sequences
$$
0\longrightarrow\cC_{\lambda+1}\longrightarrow\cB(-\lambda h_Y)^{\oplus\alpha_\lambda}\longrightarrow\cC_\lambda\longrightarrow0.
$$
Tensoring the above exact sequences by $\cO_Y((\lambda-i)h_Y)$ and $\cO_Y((\lambda-j-1)h_Y)$ and taking their cohomologies, Equalities \eqref{BUlrich} yield respectively 
\begin{gather*}
h^{i}\big(Y,\cC_\lambda((\lambda-i)h_Y)\big)\le h^{i+1}\big(Y,\cC_{\lambda+1}((\lambda+1-i-1)h_Y)\big),\\
h^{j}\big(Y,\cC_\lambda((\lambda-j-1)h_Y)\big)\le h^{j+1}\big(Y,\cC_{\lambda+1}((\lambda+1-j-2)h_Y)\big),
\end{gather*}
for each $i>0$, $j<m$ and $0\le \lambda\le N-n-1$.
Taking into account that $\cC_0=\cA\otimes_{\p N}\cB$ and $\cC_{N-n}=\cB((n-N)h_Y)^{\oplus\alpha_{N-n}}$, the above inequalities return
\begin{equation*}
\label{Dis1}
\begin{gathered}
0\le h^{i}\big(X\cap Y,\cA\otimes_{\p N}\cB(-ih_Y)\big)\le h^{i+N-n}\big(Y,\cB((n-N-i)h_Y)^{\oplus\alpha_{N-n}}\big),\\
0\le h^{j}\big(X\cap Y,\cA\otimes_{\p N}\cB(-(j+1)h_Y)\big)\le h^{j+N-n}\big(Y,\cB((n-N-j-1)h_Y)^{\oplus\alpha_{N-n}}\big),
\end{gathered}
\end{equation*}
for $i>0$ and $j+N-n<m$, i.e. $j< n+m-N$. Equalities \eqref{BUlrich} then imply that  $\cA\otimes_{\p N}\cB$ is also Ulrich by definition. 

Finally,  $\cA$ and $\cB$ are locally free along $X\cap Y$, hence the same is true for $\cA\otimes\cB$.
\qed
\medbreak

\begin{example}
\label{eGPK3}
Let  $F:=G_1\cap G_2$ be a GPK$^{\mathrm 3}$ threefold over $\bC$ (see the introduction). 

Theorem 3.5 and Corollary 4.6 of \cite{C--MR} guarantee the existence of an Ulrich bundle $\cE_i$ on $G_i$ of rank $3$. Since $\Pic(G_i)$ is generated by $\cO_{G_i}( h_{G_i})$, there  is  $\alpha_i\in\bZ$ such that $c_1(\cE_i)=\alpha_i h_{G_i}$. The rational number
$\mu(\cE_i)$ does not depend on the choice of the Ulrich bundle on $G_i$ (this is a standard property of Ulrich bundles: e.g. see \cite{C--MR}, Proposition 2.5), hence we finally obtain $\alpha_i=3$ by combining the previous discussion with Corollary 3.7 in \cite{C--MR}. 

Theorem \ref{tUlrichTensor} implies  that $F$ supports the Ulrich bundle $\cE:=\cE_1\otimes_{\p9}\cE_2$ which has rank $9$. Moreover, 
$c_1(\cE)=\rk(\cE_1)c_1(\cE_2)F+\rk(\cE_2)c_1(\cE_1)F=18h_F$.

Each  hyperplane section $S$ of $F$ supports Ulrich bundles as well. The adjunction formula implies that $S$ is canonically embedded, hence it is Ulrich--wild thanks to Corollary \ref{cGeneralType}. It is not difficult to check that  $q(S)=0$, $p_g(S)=9$, and $K_S^2=25$. In particular $S$ is not complete intersection.
\end{example}

\begin{example}
\label{eDelPezzo}
If $F$ is a del Pezzo threefold over $\bC$, $F\subseteq \p{a+1}$, $3\le a\le 8$, then it supports an Ulrich bundle of rank $2$ with $c_1(\cF)=2h_F$ (see \cite[Proposition 8]{Bea3}: see also \cite{A--C}, \cite{C--F--M1}, \cite{C--F--M2}, \cite{C--F--M3} and the methods described therein).
By combining this fact with the existence of Ulrich bundles on each hypersurface in $\Delta\subseteq\p N$ of degree $\delta$ (see \cite{B--H}), we deduce that each smooth hypersurface section $S=F\cap \Delta$ certainly supports Ulrich bundles of high rank. 

The surface $S$ has degree $\delta a$ and $K_S=(\delta-2)h_S$ by adjunction. In particular $h_SK_S=(\delta-2)\delta a$. Moreover the cohomology of 
$$
0\longrightarrow\cO_F(-\delta h_F)\longrightarrow\cO_F\longrightarrow\cO_S\longrightarrow0
$$
yields $\chi(\cO_S)=1-\chi(\cO_F(-\delta h_F))$. The Riemann--Roch theorem on $F$ and the equality $h_Fc_2(\Omega^1_{F\vert k})=12$ finally return
$$
\chi(\cO_S)=-\frac16\delta(\delta-1)(\delta-2)a+\delta.
$$
Simple computations show that Inequality \eqref{Bound1} is then equivalent to
$$
-\delta(\delta-1)(\delta-5)a>24-12\delta.
$$
When $a\ge3$, the above inequality is satisfied if and only if $2\le \delta\le 5$. Thus $S$ is Ulrich--wild in this range, thanks to Theorem \ref{tSurfaceWild}. Notice that $S$ is a surface of general type when $\delta\ge3$ and it is a $K3$ surface when $\delta=2$: the existence of special Ulrich bundles of rank $2$ on each $K3$ surface has been recently proved in \cite{Fa}. 

We spend some further words in the latter case. If the rank of $\Delta\subseteq\p{a+1}$ is $4$, then $\Delta$ is endowed with two pencils of linear spaces of dimension $a-1$. They induce two fibrations on $S$ with elliptic normal curves of degree $a$ as fibres. 

If we choose projective coordinates in $\p{a+1}$ such that $\Delta=\{\ x_0x_2-x_1x_3=0\ \}$, then the matrix
$$
\left(
\begin{array}{cc}
x_0 & x_3\\
x_1& x_2
\end{array}
\right),
$$
defines a monomorphism $\varphi\colon\cO_{\p{a+1}}(-1)^{\oplus2}\to\cO_{\p{a+1}}^{\oplus2}$, whose cokernel $\cS$ is  an Ulrich sheaf of rank $1$ on $\Delta$. Theorem \ref{tUlrichTensor} yields that $\cE:=\cF\otimes\cS$ is an Ulrich bundle of rank $2$ on $S$. Each non--zero section of $\cS$ vanishes on a linear space of one of the aforementioned pencils, thus $\cS\otimes\cO_S$ is the line bundle $\cO_S(A)$ associated to the corresponding elliptic fibration on $S$. It follows that $c_1(\cE)=2h_S+2A$. 

Notice that $A^2=0$ and $h_SA=a$, thus $c_1(\cE)^2=16a\ne18a=(3h_S)^2$: in particular $\cE$ is not special.
For further examples in this direction see \cite{C--G}.
\end{example}

\begin{remark}
\label{rNonMinimal}
Even if $\cA$ and $\cB$ are Ulrich vector bundles of minimal rank on $X$ and $Y$, the bundle $\cA\otimes\cB$ need not be of minimal rank on $X\cap Y$.

For example, let $F\subseteq\p4$ be a smooth cubic threefold and $\mathcal L\in\Pic(F)$. If $\mathcal L$ were Ulrich, then $h^0\big(F,\mathcal L(-h_F)\big)=0$ and $h^0\big(F,\mathcal L\big)\ne0$. Since the $\Pic(F)$ is generated by $\cO_F(h_F)$, it follows that $\mathcal L\cong\cO_F$, which is not Ulrich because $h^3\big(F,\cO_F(-3h_F)\big)\ne0$.

Nevertheless, for each general hyperplane $H\subseteq\p4$, the surface $S:=F\cap H$ supports Ulrich line bundles (see \cite{Bea}, Corollaries 6.4 and 1.12).
\end{remark}

\section{Ulrich wildness of some surfaces}
\label{sLowDegree}
In this section we will prove Theorems \ref{tLowDegree} and \ref{tLowWild} stated in the introduction.

\medbreak
\noindent{\it Proof of Theorem \ref{tLowDegree}.}
If $S$ is in the classes  (VIII), (XVIII), (IL) the existence of special Ulrich bundles of rank $2$ have been proved in \cite[Proposition 7.6]{C--K--M2} and \cite[Theorem 1]{Fa}. Every surface in  class (XXVIII) is the blow up of a polarized surface in class (IL) at a point. Thus \cite[Theorem 0.1]{Kim} implies that $S$ supports a  special Ulrich bundle of rank $2$. 

In the cases (XIII), (XIX), (XXX), (LI) and (LII), then $S$ is either a surface in $\p3$ of degree $5\le d\le8$, or quadro--quartic complete intersection in $\p4$. For the existence of Ulrich bundles (possibly of very high rank) in these cases see \cite{H--U--B} and the references therein. 

In \cite[Proposition 7.6]{Bea1} the author proves the existence of special Ulrich bundles of rank $2$ on the general surface in $\p3$ of degree $d\le 15$. This proves the statement in cases (XIII), (XIX), (XXX) and (LII).

We now examine the case (LI). We know the existence of a quadric $Q$ and a quartic $F$ in $\p4$ such that $S=Q\cap F$. If $F$ is general, then Proposition 8.9 of \cite{Bea1} guarantees the existence of a special Ulrich bundle $\cF$ of rank $2$ on $F$: we have $c_1(\cF)=3h_F$. If $Q$ is smooth, then it supports a unique spinor bundle $\cS$. Theorem 2.3 and Remark 2.9 of \cite{Ot} imply that $\cS(h_Q)$ is aCM, initialized and the Serre duality implies $h^3\big(Q,\cS(-2h_Q)\big)=h^0\big(Q,\cS\big)$. Thus $\cS(h_Q)$ is Ulrich: we have $c_1(\cS(h_Q))=h_Q$ thanks to Remark 2.9 of \cite{Ot}. Theorem \ref{tUlrichTensor} implies that $\cE:=\cF\otimes\cS(h_Q)$ is Ulrich on $S$ of rank $4$ with $c_1(\cE)=2c_1(\cF)S+2c_1(\cS(h_Q))S=8h_S$. Since $K_S=h_S$, it follows that $\cE$ is special.

Surfaces $S$ in classes  (I), (II), (IV), (VI), (XI), (XII), (XVI), (XVII), (XXVI), (XXVII), (XL), (XLIII), (XLVI), (XLVII) support Ulrich line bundles, because they are embedded as scroll in these cases (see \cite[Proposition 5]{Bea3}). The Bordiga (that is the surfaces in class (IX)) and Castelnuovo surfaces (surfaces in class (XIV)) also carry Ulrich line bundles (see \cite{MR--PL1}). The same is true for del Pezzo surfaces of degree $d=3,\dots,7$ (surfaces in classes (III), (V), (X), (XV), (XXV): see \cite{PL--T}). 

Surfaces in classes (VII), (XLI), (XLV) do not support Ulrich line bundles (see \cite{Cs4}, Example 2.1, \cite{PL--T}, \cite{Cs4} respectively). 

Surfaces in classes (VII), (XX), (XXI), (XXII), (XXIII), (XXIV), (XXXII), (XXXIII), (XXXIV), (XXXV), (XXXVI), (XXXVII), (XXXVIII), (XXXIX), (XLI), (XLII) have $p_g(S)=q(S)=0$. Moreover, the line bundle $\cO_S(h_S)$ is {\sl non--special} (i.e. $h^1\big(S,\cO_S(h_S)\big)=0$), as one can check by confronting the value of $N$ listed in the penultimate column in Table A with the expected dimension of the linear system. Thus the existence of a special Ulrich bundle of rank $2$ on them follows from Theorem 1.1 of \cite{Cs4}.

Surfaces $S$ in class (XLV) have $p_g(S)=0$ and $q(S)=1$. Moreover the linear system $\cO_S(h_S)$ is non--special thanks to \cite{G--P}, Proposition 3.1. The existence of a special Ulrich bundle of rank $2$ on such an $S$ follows from Theorem 1.1 of \cite{Cs5} (the case $e=1$ is also covered by \cite{A--C--MR}).

In the case (XXXI) the linear system $\cO_S(h_S)$ is special with $h^1\big(S,\cO_S(h_S)\big)=1$. The surface $S$ can be obtained in two steps as follows. 

In the first step we blow up  $\p2$ at some points $P_1,\dots, P_4$, obtaining a surface $S_1$: if ${\sigma}_1\colon {S}_1\to \p2$ is the blow up morphism, then we embed $S_1$ via $\cO_{S_1}(h_{S_1}):={\sigma}_1^*\cO_{\p2}(6)\otimes \cO_{S_1}(-2\sum_{i=1}^{4}e_i)$. The points $P_1,\dots,P_4$ are in general linear position in $\p2$, i.e. any three of them are not collinear: otherwise, there would exist an effective divisor $D$ on $S_1$ (the proper transform of the line through the three collinear points) such that $h_{S_1}D=0$, contradicting the ampleness of $\cO_S(h_{S_1})$. With this in mind it is immediate to check that $\cO_{S_1}(h_{S_1})$ is non--special. Thus $S_1$ supports a  special Ulrich bundle of rank $2$ with respect to $\cO_{S_1}(h_{S_1})$ also in this case. 

In the second step we blow up $S_1$ at some points $P_{5},\dots,P_{16}$: if $\sigma_2\colon S\to S_1$ is the blow up morphism, then we embed $S$ via $\cO_S(h_S):={\sigma}_2^*\cO_{S_1}(h_{S_1})\otimes \cO_{{S}}(-\sum_{j=5}^{16}e_j)$. Thus Theorem 0.1 of \cite{Kim} implies that $S$ supports a  special Ulrich bundle of rank $2$ also in this case.

Consider the case (XLIV). On the one hand $h^0\big(S,\cO_S(h_S)\big)=5$. On the other hand the Riemann--Roch theorem yields $\chi(\cO_S(h_S))=4$, hence $h^1\big(S,\cO_S(h_S)\big)=1$. Thus the surface is specially polarized in this case too.

Let $\mathfrak a$ be an arbitrary divisor on $C$ of degree $4+e$. Then $\cO_X(h_X):=\cO_X(2\xi+\mathfrak a f)$ is very ample on $X$, thanks to \cite[Theorem 3.3]{Ho1} when $e=1$ and \cite[Theorem 3.4]{Ho2} when $e=-1$. Moreover, it is also non--special, thanks to \cite[ Proposition 3.1]{G--P}. Thus $X$ supports a special Ulrich bundle of rank $2$ with respect to $\cO_X(h_X)$, thanks to Corollary 3.5 of \cite{Cs5} (the case $e=1$ being also covered by \cite{A--C--MR}). The existence on $S$ of a special Ulrich bundle of rank $2$ follows again from Theorem 0.1 of \cite{Kim}.

We now construct an Ulrich bundle on a surface $S$ in class (XLVIII) (see \cite{Ok3} for details). Let $Z\subseteq S$ be a set of $9$ general points. Since $\cO_S(h_S+4e_1)\cong\sigma^*\cO_X(h_X)\otimes\cO_S(2e_1)$, it follows that each divisor in  $\vert \cO_S(h_S+4e_1)\vert$ is the sum of a divisor in $\vert \sigma^*h_X\vert$ plus $2e_1$, hence 
$$
h^0\big(S,\cO_S(h_S+4e_1)\big)=h^0\big(S,\sigma^*\cO_X(h_X)\big).
$$
Moreover, $R^i\sigma_*\sigma^*\cO_X(h_X)\cong \cO_X(h_X)\otimes R^i\sigma_*\cO_S\cong 0$ if $i\ge1$ (see \cite{Ha2}, Proposition V.3.4), hence 
\begin{equation}
\label{h^0K3octic}
h^0\big(S,\sigma^*\cO_X(h_X)\big)=h^0\big(X,\cO_X(h_X)\big)=8.
\end{equation}
We deduce from Theorem 5.1.1 in \cite{H--L} that $Z$ has the Cayley--Bacharach property with respect to $\cO_S(h_S+4e_1)$, hence there is an exact sequence of the form
\begin{equation*}
\label{seqK3octic}
0\longrightarrow \cO_S\longrightarrow \cF\longrightarrow \cI_{Z\vert S}(h_S-K_S+4e_1)\longrightarrow 0
\end{equation*}
where $\cF$ is a vector bundle of rank $2$ with $c_1(\cF)=h_S-K_S+4e_1$ and $c_2(\cF)=9$. 

The bundle $\cE:=\cF(h_S+K_S-2e_1)$ fits into the exact sequence
\begin{equation}
\label{seqK3octic}
0\longrightarrow \cO_S(h_S+K_S-2e_1)\longrightarrow \cE\longrightarrow \cI_{Z\vert S}(2h_S+2e_1)\longrightarrow 0:
\end{equation}
we have $c_1(\cE)=3h_S+K_S$, whence $h^0\big(S,\cE(-h_S)\big)=h^0\big(S,\cE^\vee(2h_S+K_S)\big)$
and $c_2(\cE)=27$. Moreover, computing the cohomology of Sequence \eqref{seqK3octic} tensored by $\cO_S(-h_S)$, taking into account the isomorphisms $\cO_S(K_S)=\cO_S(e_1)$, $\sigma^*\cO_X(h_X)\cong\cO_S(h_S+2e_1)$ and Equality \eqref{h^0K3octic}, we finally obtain
$$
h^0\big(S,\cE(-h_S)\big)\le h^0\big(S,\cO_S(-e_1)\big)+h^0\big(S,\cI_{Z\vert S}\otimes\cO_X(h_X)\big)=0,
$$
because $\deg(Z)=9$ and $h^0\big(S,\sigma^*\cO_X(h_X)\big)=8$. Proposition \ref{pUlrich} yields that $\cE$ is a special Ulrich bundle of rank $2$ on $S$. 
\qed
\medbreak

\begin{remark}
\label{kappa=1}
If $S$ is in classes (XXIX) and (L), then its canonical map is an elliptic fibration $\epsilon\colon S\to \p1$, fibres being elliptic normal curves of  degrees $3$ and $4$.  

In \cite[Theorem III.4.2 and Observation III.3.5]{Lop}, the author proves that $\Pic(S)$ is generated by $h_S$ and $K_S$  for the very general surface $S$ as above. The map $\epsilon$ has no sections, otherwise there would exist integers $x,y$ such that $(xK_S+yh_S)K_S=1$, contradicting the conditions $K^2_{S}=0$ and $3\le h_SK_S\le 4$. 
Thus, we cannot use results from \cite{MR--PL2} for proving the existence of Ulrich bundles on such surfaces.
\end{remark}

We are now ready to prove Theorem \ref{tLowWild} as an easy corollary of Theorem \ref{tSurfaceWild}.

\medbreak
\noindent{\it Proof of Theorem \ref{tLowWild}.}
Surfaces $S$ in classes (III), (V), (IX), (X), (XI), (XIV), (XV), (XVI), (XX), (XXI), (XXII), (XXIII), (XXIV), (XXV), (XXVI), (XXXII), (XXXIII), (XXXIV), (XXXV), (XXXVI), (XXXVII), (XXXVIII), (XXXIX), (XL), (XLI), (XLII) are non--special and have $p_g(S)=q(S)=0$. Thus their Ulrich--wildness follows from Theorem 1.3 of \cite{Cs4}. By the same theorem, if $d\le 4$ and $\pi(S)$  vanishes, then $S$ is not Ulrich--wild: these cases are (I), (II), (IV), (VI), (VII).

Similarly surfaces $S$ in classes (XII), (XVII), (XXVII), (XLIII), (XLV) are Ulrich--wild, thanks to \cite[Theorem 1.3]{Cs5} and \cite[Proposition 3.1]{G--P}.

In cases (VIII), (XIII), (XVIII), (XIX), (XXVIII), (XXX), (XXXI), (XLIV), (XLVIII), (IL), (LI), (LII) 
the statement follows easily from Theorem \ref{tSurfaceWild}, thanks to the invariants listed in the table.

In cases (XLV), (XLVI) the surface $S$ is geometrically ruled on a curve $C$ and it is embedded in $\p N$ as a scroll. Following the notation in Remark \ref{rTable}, we know that the invariant $e$ is $-2$ and $0$ in cases (XLV) and (XLVI) respectively. Thus, we have $\cO_S(h_S)\cong \cO_S(\xi+p^*\mathfrak b)$, hence $2\deg(\mathfrak b)-e=\deg(S)=8$.

Assertion 2) of Proposition 5 in \cite{Bea3} implies that for each general $\mathfrak u\in\Pic^{g-1}(C)$ then $\mathcal L:=\cO_S(h_S+p^*\mathfrak u)\cong \cO_S(\xi+p^*\mathfrak{b}+p^*\mathfrak u)$ is Ulrich. It follows from Proposition \ref{pUlrich} that $\cM:=\cO_S(2h_S+K_S-p^*\mathfrak u)\cong p^*\cO_E(2\mathfrak b+\mathfrak h+\mathfrak k-\mathfrak u)$ is Ulrich too. 

Such bundles are trivially simple and $h^0\big(S,\mathcal L\otimes\cM^\vee\big)=h^0\big(S,\cM\otimes\mathcal L^\vee\big)=0$ because $\mathcal L\not\cong\cM$. Since $\mathcal L\otimes\cM^\vee\cong\cO_S(\xi-p^*\mathfrak b-p^*\mathfrak h-p^*\mathfrak k+2\mathfrak u)$, it follows from Equality \eqref{RRGeneral} that
$$
h^1\big(S,\mathcal L\otimes\cM^\vee\big)\ge-\chi(\mathcal L\otimes\cM^\vee)=2\deg(\mathfrak b)-e=8.
$$
The statement thus follows from Theorem \ref{tFPL}.
\qed
\medbreak

\begin{remark}
The proofs of Theorems 1.3 of \cite{Cs4} and \cite{Cs5} used above contain a gap which can be overcome by assuming that $k$ is uncountable (see the erratum). Thus, such theorems certainly hold when $k=\bC$ as we assume in Theorem \ref{tLowWild}. 
\end{remark}
\medbreak
\centerline{\textsc{Acknowledgement}}
\smallbreak

The author is indebted with the referee for her/his criticisms, questions, remarks and suggestions which have considerably improved the whole exposition.


\begin{thebibliography}{44}

\bibitem{A--D--S}
H. Abo, W. Decker, N. Sasakura: \emph{An elliptic conic bundle in $\bP^4$ arising from a stable rank--$3$ vector bundle}. Math. Z. \textbf{229} (1998), 725--741.

\bibitem{A--C--MR}
M. Aprodu, L. Costa, R.M. Mir\'o--Roig: \emph{Ulrich bundles on ruled surfaces}. J. Pure Appl. Algebra, \textbf{222} (2018), 131--138.

\bibitem{A--C}
E. Arrondo, L. Costa: \emph{Vector bundles on Fano $3$--folds without intermediate cohomology}. Comm. Algebra  \textbf{ 28} \rm (2000),  3899--3911.

\bibitem{A--F--O}
M. Aprodu, G. Farkas, A. Ortega: \emph{Minimal resolutions, Chow forms and Ulrich bundles on $K3$ surfaces}. J. Reine Angew. Math. \textbf{730} (2017), 225--249.

\bibitem{B--H--P--VV}
W. Barth, K. Hulek, Ch. Peters, A. van de Ven: \emph{Compact complex surfaces}. Second edition. Springer \rm(2004). 

\bibitem{Bea}
  A. Beauville: \emph{Determinantal hypersurfaces}.  Michigan Math. J. \textbf{48} \rm(2000), 39--64.

\bibitem{Bea1}
A. Beauville: \emph{Ulrich bundles on abelian surfaces}.  Proc. Amer. Math. Soc. \textbf{144} (2016), 4609--4611.
 
 
\bibitem{Bea3}
A. Beauville: \emph{An introduction to Ulrich bundles}. Eur. J. Math., \textbf{4} (2018), 26--36.

\bibitem{B--H}
J. Backelin, J. Herzog: \emph{On Ulrich-modules over hypersurface rings}. In Commutative algebra (Berkeley, CA, 1987), 63--68, Math. Sci. Res. Inst. Publ., 15, Springer, New York, 1989.

\bibitem{B--C--P}
L. Borisov, A. C\u{a}ld\u{a}raru, A. Perry: \emph{Intersections of two grassmannianns in $\p9$}. Available at arXiv:1707.00534 [math.AG]; to appear in J. Reine Angew. Math..

\bibitem{B--S}
M. Brodmann, P. Schenzel: \emph{Arithmetic properties of projective varieties of almost minimal
degree}. J. Algebraic Geom. \textbf{16} (2007), 347--400.

\bibitem{C--H2}
  M. Casanellas, R. Hartshorne, F. Geiss, F.O. Schreyer: \emph{Stable Ulrich bundles}. Int. J. of Math. \textbf{23} \rm(2012), 1250083. 

\bibitem{Cs4}
G. Casnati: \emph{Special Ulrich bundles on non--special surfaces with $p_g=q=0$}. Int. J. Math. \textbf{28} (2017), 1750061. \emph{Erratum: \lq\lq Special Ulrich bundles on non-special surfaces with $p_g=q=0$\rq\rq}. Int. J. of Math. \textbf{29} (2018), 1892001.

\bibitem{Cs5}
G. Casnati: \emph{Ulrich bundles on non--special surfaces with $p_g=0$ and $q=1$}. Rev. Mat.
Complut. \textbf{32} (2019), 559--574.   \emph{Correction to: Ulrich bundles on non-special surfaces with $p_g=0$ and $q=1$}. Rev. Mat.
Complut. \textbf{32} (2019) 575--577.

\bibitem{C--F--M1}
G. Casnati, D. Faenzi, F. Malaspina: \emph{Rank two aCM bundles on the del Pezzo threefold with Picard number $3$}. J. Algebra, \textbf{429} (2015), 413--446.

\bibitem{C--F--M2}
G. Casnati, D. Faenzi, F. Malaspina: \emph{Rank two aCM bundles on the del Pezzo fourfold of degree $6$ and its general hyperplane section}. J. Pure Appl. Algebra, \textbf{22} (2018), 585--609.

\bibitem{C--F--M3}
G. Casnati, M. Filip, F. Malaspina: \emph{Rank two aCM bundles on the del Pezzo threefold of degree 7}. Rev. Mat. Complut., \textbf{30} (2017), 129--165.

\bibitem{C--G}
G. Casnati, F. Galluzzi: \emph{Stability of rank $2$ Ulrich bundles on projective $K3$ surface}. Math. Scand., \textbf{122} (2018), 239--256.

\bibitem{C--K--M2}
  E. Coskun, R.S. Kulkarni, Y. Mustopa: \emph{Pfaffian quartic surfaces and representations of Clifford algebras}. Doc. Math.  \textbf{17} \rm (2012), 1003--1028.
 
\bibitem{C--MR}
L. Costa, R.M. Mir\'o--Roig: \emph{$GL(V)$--invariant Ulrich bundles on grassmannians}.  Math. Ann. \textbf{361} (2015), 443--457.
  
\bibitem{De--Bea}
O. Debarre: \emph{In\'egalit\'es num\'eriques pour les surfaces de type g\'en\'eral}. Appendice: \emph{\lq L'in\'egalit\'e $p_g\ge2q-4$ pour les surfaces de type  g\'en\'eral\rq} par A. Beauville.
 Bull. Soc. Math. France \textbf{110} (1982), 3, 319--346.

\bibitem{E--S--W}
D. Eisenbud, F.O. Schreyer, J. Weyman: \emph{Resultants and Chow forms via exterior syzigies}.  J. Amer. Math. Soc. \textbf{16} \rm(2003), 537--579.

\bibitem{Fa}
D. Faenzi: \emph{Ulrich sheaves on $K3$ surfaces}. Available at arXiv:1807.07826 [math.AG].

\bibitem{F--PL}
D. Faenzi, J. Pons-Llopis: \emph{The CM representation type of projective varieties}. Available at arXiv:1504.03819v2 [math.AG].

\bibitem{G--P}
F.J. Gallego, B. P. Purnaprajna: \emph{Normal presentation on elliptic ruled surfaces}. J. Algebra \textbf{186} (1996), 597--625.

\bibitem{Ha2}
  R. Hartshorne: \emph{Algebraic geometry}. G.T.M. 52, Springer \rm (1977).

\bibitem{H--U--B}
J. Herzog, B. Ulrich, J. Backelin: \emph{Linear maximal Cohen--Macaulay modules over strict complete intersections}. J. Pure Appl. Algebra \textbf{71} (1991), 187--202.

\bibitem{Ho1}
Y. Homma: \emph{Projective normality and the defining equations of ample invertible
sheaves on elliptic ruled surfaces with $e\ge0$}. Natur. Sci. Rep. Ochanomizu Univ. \textbf{31} (1980), 61--73.

\bibitem{Ho2}
Y. Homma: \emph{Projective normality and the defining equations of an elliptic ruled
surface with negative invariant}. Natur. Sci. Rep. Ochanomizu Univ. \textbf{33} (1982), 17--26.

\bibitem{H--L}
D. Huybrechts, M. Lehn: \emph {The geometry of moduli spaces of sheaves. Second edition}. Cambridge Mathematical Library, Cambridge U.P. \rm (2010).

\bibitem{Io2}
P. Ionescu: \emph{Embedded projective varieties of small invariants}. In \lq Algebraic geometry, Bucharest 1982 (Bucharest, 1982)\rq, L.N.M. \textbf{1056} (1984), 142--186.

\bibitem{Io3}
P. Ionescu: \emph{Embedded projective varieties of small invariants III}. In \lq Algebraic geometry (L'Aquila, 1988)\rq, L.N.M. \textbf{1417} (1990), 138--154.

\bibitem{Kim}
Y. Kim: \emph{Ulrich bundles on blowing ups}. C.R.Acad.Sci. Paris, \textbf{354} (2016), 1215--1218. 

\bibitem{Lop}
A. Lopez: \emph{Noether--Lefschetz theory and the Picard group of projective surfaces}.  Memoirs Amer. Math. Soc.  \textbf{89} \rm  (1991).

\bibitem{MR--PL1}
R.M. Mir\'o--Roig, J. Pons--Llopis: \emph{Representation Type of Rational ACM Surfaces $X\subseteq\p4$}. Algebr. Represent. Theor. \textbf{16} \rm (2013), 1135--1157.

\bibitem{MR--PL}
R.M. Mir\'o--Roig, J. Pons-Llopis: \emph{$n$--Dimensional Fano varieties of  wild representation type}.  J. Pure Appl. Algebra \textbf{218} (2014), 1867--1884.

\bibitem{MR--PL2}
R.M. Mir\'o--Roig, J. Pons--Llopis: \emph{Special Ulrich bundles on elliptic surfaces}. Preprint.

\bibitem{No}
D.G. Northcott: \emph{A first course of homological algebra}. Cambridge, Cambridge U.P. \rm (1977).

\bibitem{Ok3}
Ch. Okonek: \emph{Fl\"achen vom Grad $8$ im  $\p4$}. Math. Z. \textbf{191} (1986), 207--223.

\bibitem{Ot}
  G. Ottaviani: \emph{Spinor bundles on quadrics}. Trans. A.M.S. \textbf{ 307} \rm(1988), 301--316.
  
\bibitem{PL--T}
J. Pons-Llopis, F. Tonini: \emph{ACM bundles on del Pezzo surfaces}. Matematiche (Catania) \textbf{64} (2009), 177--211.

\bibitem{Ra}
K. Ranestad: \emph{On smooth surfaces of degree ten in the projective fourspace}.
thesis, Oslo, (1988).



\end{thebibliography}
\end{document}